\documentclass[a4paper,12pt]{amsart}
\usepackage{amsfonts}
\usepackage{amsmath}
\usepackage{hyperref}
\usepackage{amssymb}
\usepackage{color}

\allowdisplaybreaks
\numberwithin{equation}{section}
\newtheorem{theorem}{Theorem}[section]
\newtheorem{proposition}[theorem]{Proposition}

\newtheorem{corollary}[theorem]{Corollary}

\theoremstyle{definition}

\theoremstyle{remark}
\newtheorem{remark}[theorem]{Remark}

\input{tcilatex}

\begin{document}
\title[the Toeplitz algebra of a higher-rank graph]{Realising the Toeplitz
algebra of a higher-rank graph as a Cuntz-Krieger algebra}
\author{Yosafat E. P. Pangalela}
\address{Department of Mathematics and Statistics, University of Otago, PO
Box 56, Dunedin 9054, New Zealand}
\email{yosafat.pangalela@maths.otago.ac.nz}
\thanks{This research is part of the author's Ph.D. thesis, supervised by
Professor Iain Raeburn and Dr. Lisa Orloff Clark.}
\date{April 29, 2015}

\begin{abstract}
For a row-finite higher-rank graph $\Lambda $, we construct a higher-rank
graph $T\Lambda $ such that the Toeplitz algebra of $\Lambda $ is isomorphic
to the Cuntz-Krieger algebra of $T\Lambda $. We then prove that the
higher-rank graph $T\Lambda $ is always aperiodic and use this fact to give
another proof of a uniqueness theorem for the Toeplitz algebras of
higher-rank graphs.
\end{abstract}

\maketitle

\section{Introduction}

Higher-rank graphs and their Cuntz-Krieger algebras were introduced by
Kumjian and Pask in \cite{KP00} as a generalisation of the Cuntz-Krieger
algebras of directed graphs. Kumjian and Pask proved an analogue of the
Cuntz-Krieger uniqueness theorem for a family of \emph{aperiodic }%
higher-rank graphs \cite[Theorem 4.6]{KP00}. Aperiodicity is a
generalisation of Condition (L) for directed graphs and comes in several
forms for different kinds of higher-rank graphs (see \cite{FMY05,KP00, LS10,
RSY03, RSY04,RS07,RS09,S12}).

The Toeplitz algebra of a directed graph is an extension of the
Cuntz-Krieger algebra in which the Cuntz-Krieger equations at vertices are
replaced by inequalities. An analogous family of Toeplitz algebras for
higher-rank graph was introduced and studied by Raeburn and Sims \cite{RS05}%
. They proved a uniqueness theorem for Toeplitz algebras \cite[Theorem 8.1]%
{RS05}, generalising a previous theorem for directed graphs \cite[Theorem 4.1%
]{FR99}.

For a directed graph $E$, the Toeplitz algebra of $E$ is canonically
isomorphic to the Cuntz-Krieger algebra of a graph $TE$ (see \cite[Theorem
3.7]{MT04} and \cite[Lemma 3.5]{S10}). Here we provide an analogous
construction for a row-finite higher-rank graph $\Lambda $. We build a
higher-rank graph $T\Lambda $, and show that the Toeplitz algebra of $%
\Lambda $ is canonically isomorphic to the Cuntz-Krieger algebra of $%
T\Lambda $ (Theorem \ref{isomorphism}). Our proof relies on the uniqueness
theorem of \cite{RS05}. However, it is interesting to observe that the
higher-rank graph $T\Lambda $ is always aperiodic. Hence our isomorphism
shows that the uniqueness theorem of \cite{RS05} is a consequence of the
general Cuntz-Krieger uniqueness theorem of \cite{RSY04} (see Remark \ref%
{isomorphism-other-direction}).

\section{Higher-rank graphs}

\label{k-graph}

Let $k$ be a positive integer. We regard $\mathbb{N}^{k}$ as an additive
semigroup with identity $0$. For $m,n\in \mathbb{N}^{k}$, we write $m\vee n$
for \ their coordinate-wise maximum.

A \emph{higher-rank graph} or $k$\emph{-graph} is a pair $\left( \Lambda
,d\right) $ consisting of a countable small category $\Lambda $ together
with a functor $d:\Lambda \rightarrow \mathbb{N}^{k}$ satisfying the \emph{%
factorisation property}: for every $\lambda \in \Lambda $ and $m,n\in 
\mathbb{N}^{k}$ with $d\left( \lambda \right) =m+n$, there are unique
elements $\mu ,\nu \in \Lambda $ such that $\lambda =\mu \upsilon $ and $%
d\left( \mu \right) =m$, $d\left( \nu \right) =n$. We then write $\lambda
\left( 0,m\right) $ for $\mu $ and $\lambda \left( m,m+n\right) $ for $\nu $%
. We regard elements of $\Lambda ^{0}$ as \emph{vertices }and elements of $%
\Lambda $ as \emph{paths}. For detailed explanation and examples, see \cite[%
Chapter 10]{CBMS}.

For $v\in \Lambda ^{0}$ and $E\subseteq \Lambda $, we define $vE:=\left\{
\lambda \in E:r\left( \lambda \right) =v\right\} $ and $m\in \mathbb{N}^{k}$%
, we write $\Lambda ^{m}:=\left\{ \lambda \in \Lambda :d\left( \lambda
\right) =m\right\} $.We use term \emph{edge} to denote a path $e\in \Lambda
^{e_{i}}$ where $1\leq i\leq k$, and write 
\begin{equation*}
\Lambda ^{1}:=\bigcup_{1\leq i\leq k}\Lambda ^{e_{i}}
\end{equation*}%
for the set of all edges. We say that $\Lambda $ is \emph{row-finite} if for
every $v\in \Lambda ^{0}$, the set $v\Lambda ^{e_{i}}$ is finite for $1\leq
i\leq k$. Finally, we say $v\in \Lambda ^{0}$ is a \emph{source} if there
exists $m\in \mathbb{N}^{k}$ such that $v\Lambda ^{m}=\emptyset $.

For a row-finite $k$-graph $\Lambda $, we shall construct a $k$-graph $%
T\Lambda $ which is row-finite and always has sources. Our $k$-graph $%
T\Lambda $ is typically not \emph{locally convex }in the sense of \cite[%
Definition 3.9]{RSY03} (see Remark \ref{locally-convex}), so the approriate
definition of Cuntz-Krieger $\Lambda $-family is the one in \cite{RSY04}.
For detailed discussion about row-finite $k$-graphs and their
generalisations, see \cite[Section 2]{W11}.

From now on, we focus on a row-finite $k$-graph $\Lambda $. For $\lambda
,\mu \in \Lambda $, we say that $\tau $ is a \emph{minimal common extension }%
of $\lambda $ and $\mu $ if 
\begin{equation*}
d\left( \tau \right) =d\left( \lambda \right) \vee d\left( \mu \right) \text{%
, }\tau \left( 0,d\left( \lambda \right) \right) =\lambda \text{ and }\tau
\left( 0,d\left( \mu \right) \right) =\mu \text{.}
\end{equation*}
Let $\func{MCE}\left( \lambda ,\mu \right) $ denote the collection of all
minimal common extensions of $\lambda $ and $\mu $. Then we write%
\begin{equation*}
\Lambda ^{\min }\left( \lambda ,\mu \right) :=\left\{ \left( \lambda
^{\prime },\mu ^{\prime }\right) \in \Lambda \times \Lambda :\lambda \lambda
^{\prime }=\mu \mu ^{\prime }\in \func{MCE}\left( \lambda ,\mu \right)
\right\} \text{.}
\end{equation*}%
A set $E\subseteq v\Lambda ^{1}$ is \emph{exhaustive} if for all $\lambda
\in v\Lambda $, there exists $e\in E$ such that $\Lambda ^{\min }\left(
\lambda ,e\right) \neq \emptyset $.

A \emph{Toeplitz-Cuntz-Krieger }$\Lambda $\emph{-family} is a collection $%
\left\{ t_{\lambda }:\lambda \in \Lambda \right\} $ of partial isometries in
a $C^{\ast }$-algebra $B$ satisfying:

\begin{enumerate}
\item[(TCK1)] $\left\{ t_{v}:v\in \Lambda ^{0}\right\} $ is a collection of
mutually orthogonal projections;

\item[(TCK2)] $t_{\lambda }t_{\mu }=t_{\lambda \mu }$ whenever $s\left(
\lambda \right) =r\left( \mu \right) $; and

\item[(TCK3)] $t_{\lambda }^{\ast }t_{\mu }=\sum_{(\lambda ^{\prime },\mu
^{\prime })\in \Lambda ^{\min }\left( \lambda ,\mu \right) }t_{\lambda
^{\prime }}t_{\mu ^{\prime }}^{\ast }$ for all $\lambda ,\mu \in \Lambda $.
\end{enumerate}

\begin{remark}
In \cite[Lemma 9.2]{RS05}, Raeburn and Sims required also that
\textquotedblleft for all $m\in \mathbb{N}^{k}\backslash \left\{ 0\right\} $%
, $v\in \Lambda ^{0}$, and every set $E\subseteq v\Lambda ^{m}$, $t_{v}\geq
\sum_{\lambda \in E}t_{\lambda }t_{\lambda }^{\ast }$\textquotedblright .
However, by \cite[Lemma 2.7 (iii)]{RSY04}, this follows from (TCK1-3), and
hence our definition is basically same as that of \cite{RS05}.
\end{remark}

Meanwhile, based on \cite[Proposition C.3]{RSY04}, a \emph{Cuntz-Krieger }$%
\Lambda $\emph{-family} is a Toeplitz-Cuntz-Krieger $\Lambda $-family $%
\left\{ t_{\lambda }:\lambda \in \Lambda \right\} $ which satisfies

\begin{enumerate}
\item[(CK)] $\prod_{e\in E}\left( t_{r\left( E\right) }-t_{e}t_{e}^{\ast
}\right) =0$ for all $v\in \Lambda ^{0}$ and exhaustive $E\subseteq v\Lambda
^{1}$.
\end{enumerate}

Raeburn and Sims proved in \cite[Section 4]{RS05} that there is a $C^{\ast }$%
-algebra $TC^{\ast }\left( \Lambda \right) $ generated by a universal
Toeplitz-Cuntz-Krieger $\Lambda $-family $\left\{ t_{\lambda }:\lambda \in
\Lambda \right\} $. If $\left\{ T_{\lambda }:\lambda \in \Lambda \right\} $
is a Toeplitz-Cuntz-Krieger $\Lambda $-family in $C^{\ast }$-algebra $B$, we
write $\phi _{T}$ for the homomorphism of $TC^{\ast }\left( \Lambda \right) $
into $B$ such that $\phi _{T}\left( t_{\lambda }\right) =T_{\lambda }$ for $%
\lambda \in \Lambda $. The quotient of $TC^{\ast }\left( \Lambda \right) $
by the ideal generated by%
\begin{equation*}
\{\prod_{e\in E}\left( t_{r\left( E\right) }-t_{e}t_{e}^{\ast }\right)
=0:v\in \Lambda ^{0},E\subseteq v\Lambda ^{1}\text{ is exhaustive}\}
\end{equation*}%
is generated by a universal family of Cuntz-Krieger $\Lambda $-family $%
\left\{ s_{\lambda }:\lambda \in \Lambda \right\} $, and hence we can
identify it with the $C^{\ast }$-algebra $C^{\ast }\left( \Lambda \right) $.
For a Cuntz-Krieger $\Lambda $-family $\left\{ S_{\lambda }:\lambda \in
\Lambda \right\} $ in $C^{\ast }$-algebra $B$, we write $\pi _{S}$ for the
homorphism of $C^{\ast }\left( \Lambda \right) $ into $B$ such that $\pi
_{S}\left( s_{\lambda }\right) =S_{\lambda }$ for $\lambda \in \Lambda $.
Furthermore, we have $s_{v}\neq 0$ for $v\in \Lambda ^{0}$ \cite[Proposition
2.12]{RSY04}.

As for directed graphs, we have uniqueness theorems for the Toeplitz algebra 
\cite[Theorem 8.1]{RS05} and the Cuntz-Krieger algebra \cite[Theorem 4.5]%
{RSY04}. The former does not need any hypothesis on the $k$-graph as stated
in the following theorem.

\begin{theorem}
\label{uniqueness-theorem-of-toeplitz}Let $\Lambda $ be a row-finite $k$%
-graph. Let $\left\{ T_{\lambda }:\lambda \in \Lambda \right\} $ be a
Toeplitz-Cuntz-Krieger $\Lambda $-family in a $C^{\ast }$-algebra $B$.
Suppose that for every $v\in \Lambda ^{0}$,%
\begin{equation}
\prod_{e\in v\Lambda ^{1}}\left( T_{v}-T_{e}T_{e}^{\ast }\right) \neq 0 
\tag{*}
\end{equation}%
(where this includes $T_{v}\neq 0$ if $v\Lambda ^{1}=\emptyset $). Suppose
that $\phi _{T}:TC^{\ast }\left( \Lambda \right) \rightarrow B$ is the
homomorphism such that $\phi _{T}\left( t_{\lambda }\right) =T_{\lambda }$
for $\lambda \in \Lambda $. Then $\phi _{T}:TC^{\ast }\left( \Lambda \right)
\rightarrow B$ is injective.
\end{theorem}

\begin{remark}
Every $k$-graph $\Lambda $ gives a product system of graphs over $\mathbb{N}%
^{k}$ and a Toeplitz-Cuntz-Krieger $\Lambda $-family gives a Toeplitz $%
\Lambda $-family of the product system \cite[Lemma 9.2]{RS05}. Lemma 9.3 of 
\cite{RS05} shows that, if the Toeplitz-Cuntz-Krieger $\Lambda $-family
satisfies (*), then the Toeplitz $\Lambda $-family satisfies the hypothesis
of \cite[Theorem 8.1]{RS05}.
\end{remark}

\begin{remark}
In the actual hypothesis, we need to verify whether $\prod_{1\leq i\leq
k}(T_{v}-\sum_{e\in G_{i}}T_{e}T_{e}^{\ast })\neq 0$ for every $v\in \Lambda
^{0}$, $1\leq i\leq k$, and finite set $G_{i}\subseteq v\Lambda ^{e_{i}}$.
However, since we only consider row-finite $k$-graphs, then for every $v\in
\Lambda ^{0}$ and $1\leq i\leq k$, the set $v\Lambda ^{e_{i}}$ is finite.
Thus for a row finite $k$-graph, we can simplify Lemma 9.3 of \cite{RS05} as
Theorem \ref{uniqueness-theorem-of-toeplitz}.
\end{remark}

On the other hand, the Cuntz-Krieger uniqueness theorem only holds for $k$%
-graphs\ satisfying a special Condition (B) \cite{RSY04}. Later, Lewin and
Sims in \cite[Proposition 3.6]{LS10} proved that Condition (B) is equivalent
to the following \emph{aperiodicity} condition: for every pair of distinct
paths $\lambda ,\mu \in \Lambda $ with $s\left( \lambda \right) =s\left( \mu
\right) $, there exists $\eta \in s\left( \lambda \right) \Lambda $ such
that $\func{MCE}\left( \lambda \eta ,\mu \eta \right) =\emptyset $ \cite[%
Definition 3.1]{LS10}. (For discussion about the equivalence of various
aperiodicity definitions, see \cite{LS10, RS07, RS09,S12}.) Therefore we get
the following theorem.

\begin{theorem}[{\protect\cite[Theorem 4.5]{RSY04}}]
\label{cuntz-krieger-uniqueness-theorem}Suppose that $\Lambda $ is aperiodic
row-finite $k$-graph and $\left\{ S_{\lambda }:\lambda \in \Lambda \right\} $
\ is a Cuntz-Krieger $\Lambda $-family in a $C^{\ast }$-algebra $B$ such
that $S_{v}\neq 0$ for $v\in \Lambda ^{0}$. Suppose that $\pi _{S}:C^{\ast
}\left( \Lambda \right) \rightarrow B$ is the homomorphism such that $\pi
_{S}\left( s_{\lambda }\right) =S_{\lambda }$ for $\lambda \in \Lambda $ for 
$\lambda \in \Lambda $. Then $\pi _{S}$ is an injective homomorphism.
\end{theorem}

\section{The $k$-graph $T\Lambda $}

\label{tlambda-and-lambda}

Suppose that $\Lambda $ is a row-finite $k$-graph $\Lambda $. In this
section, we define a $k$-graph $T\Lambda $; later we show that $TC^{\ast
}\left( \Lambda \right) \cong C^{\ast }\left( T\Lambda \right) $ (Theorem %
\ref{isomorphism}). Interestingly, our $k$-graph $T\Lambda $ is always
aperiodic (Proposition \ref{characterisation-of-tlambda}).

\begin{proposition}
\label{tlambda}Let $\Lambda =\left( \Lambda ,d,r,s\right) $ be a row-finite $%
k$-graph. Then define sets $T\Lambda ^{0}$ and $T\Lambda $ as follows:%
\begin{equation*}
T\Lambda ^{0}:=\left\{ \alpha \left( v\right) :v\in \Lambda ^{0}\right\}
\cup \left\{ \beta \left( v\right) :v\Lambda ^{1}\neq \emptyset \right\} 
\text{;}
\end{equation*}%
\begin{equation*}
T\Lambda :=\left\{ \alpha \left( \lambda \right) :\lambda \in \Lambda
\right\} \cup \left\{ \beta \left( \lambda \right) :\lambda \in \Lambda
,s\left( \lambda \right) \Lambda ^{1}\neq \emptyset \right\} \text{.}
\end{equation*}%
Define functions $r,s:T\Lambda \backslash T\Lambda ^{0}\rightarrow T\Lambda
^{0}$ by%
\begin{align*}
r\left( \alpha \left( \lambda \right) \right) =\alpha \left( r\left( \lambda
\right) \right) &\text{, }s\left( \alpha \left( \lambda \right) \right)
=\alpha \left( s\left( \lambda \right) \right) \text{,} \\
r\left( \beta \left( \lambda \right) \right) =\alpha \left( r\left( \lambda
\right) \right) &\text{, }s\left( \beta \left( \lambda \right) \right)
=\beta \left( s\left( \lambda \right) \right)
\end{align*}%
($r,s$ are the identity on $T\Lambda ^{0}$). We also define a partially
defined product $\left( \tau ,\omega \right) \mapsto \tau \omega $ from 
\begin{equation*}
\left\{ \left( \tau ,\omega \right) \in T\Lambda \times T\Lambda :s\left(
\tau \right) =r\left( \omega \right) \right\}
\end{equation*}%
to $T\Lambda $, where%
\begin{equation*}
\left( \alpha \left( \lambda \right) ,\alpha \left( \mu \right) \right)
\mapsto \alpha \left( \lambda \mu \right)
\end{equation*}%
\begin{equation*}
\left( \alpha \left( \lambda \right) ,\beta \left( \mu \right) \right)
\mapsto \beta \left( \lambda \mu \right)
\end{equation*}%
and a function $d:T\Lambda \rightarrow \mathbb{N}^{k}$ where%
\begin{equation*}
d\left( \alpha \left( \lambda \right) \right) =d\left( \beta \left( \lambda
\right) \right) =d\left( \lambda \right) \text{.}
\end{equation*}%
Then $\left( T\Lambda ,d\right) $ is a $k$-graph.
\end{proposition}

\begin{proof}
First we claim that $T\Lambda $ is a countable category. Note that $T\Lambda 
$ is a countable since $\Lambda $ is countable.

Now we show that for all paths $\eta ,\tau ,\omega $ in $T\Lambda $ where $%
s\left( \eta \right) =r\left( \tau \right) $ and $s\left( \tau \right)
=r\left( \omega \right) $, we have $s\left( \tau \omega \right) =s\left(
\omega \right) $, $r\left( \tau \omega \right) =r\left( \tau \right) $, and $%
\left( \eta \tau \right) \omega =\eta \left( \tau \omega \right) $. If one
of $\tau ,\omega $ is a vertex then we are done. So assume otherwise, and we
have $\eta =\alpha \left( \lambda \right) $, $\tau =\alpha \left( \mu
\right) $, and $\omega $ is either $\alpha \left( \nu \right) $ or $\beta
\left( \nu \right) $ for some paths $\lambda ,\mu ,\nu $ in $\Lambda $. In
both cases, we always have $s\left( \lambda \right) =r\left( \mu \right) $, $%
s\left( \mu \right) =r\left( \nu \right) $, and $\left( \lambda \mu \right)
\nu =\lambda \left( \mu \nu \right) $. If $\omega =\alpha \left( \nu \right) 
$, we have%
\begin{equation*}
s\left( \tau \omega \right) =s\left( \alpha \left( \mu \right) \alpha \left(
\nu \right) \right) =s\left( \alpha \left( \mu \nu \right) \right) =\alpha
\left( s\left( \mu \nu \right) \right) =\alpha \left( s\left( \nu \right)
\right) =s\left( \alpha \left( \nu \right) \right) =s\left( \omega \right) 
\text{,}
\end{equation*}%
\begin{equation*}
r\left( \tau \omega \right) =r\left( \alpha \left( \mu \right) \alpha \left(
\nu \right) \right) =r\left( \alpha \left( \mu \nu \right) \right) =\alpha
\left( r\left( \mu \nu \right) \right) =\alpha \left( r\left( \mu \right)
\right) =r\left( \alpha \left( \mu \right) \right) =r\left( \tau \right) 
\text{, and}
\end{equation*}%
\begin{align*}
\left( \eta \tau \right) \omega & =\big(\alpha \left( \lambda \right) \alpha
\left( \mu \right) \big)\alpha \left( \nu \right) =\alpha \left( \lambda \mu
\right) \alpha \left( \nu \right) =\alpha \left( \left( \lambda \mu \right)
\nu \right) \\
& =\alpha \left( \lambda \left( \mu \nu \right) \right) =\alpha \left(
\lambda \right) \alpha \left( \mu \nu \right) =\alpha \left( \lambda \right) %
\big(\alpha \left( \mu \right) \alpha \left( \nu \right) \big)=\eta \left(
\tau \omega \right) \text{.}
\end{align*}%
On the other hand, if $\omega =\beta \left( \nu \right) $, then%
\begin{equation*}
s\left( \tau \omega \right) =s\left( \alpha \left( \mu \right) \beta \left(
\nu \right) \right) =s\left( \beta \left( \mu \nu \right) \right) =\beta
\left( s\left( \mu \nu \right) \right) =\beta \left( s\left( \nu \right)
\right) =s\left( \beta \left( \nu \right) \right) =s\left( \omega \right) 
\text{,}
\end{equation*}%
\begin{equation*}
r\left( \tau \omega \right) =r\left( \alpha \left( \mu \right) \beta \left(
\nu \right) \right) =r\left( \beta \left( \mu \nu \right) \right) =\alpha
\left( r\left( \mu \nu \right) \right) =\alpha \left( r\left( \mu \right)
\right) =r\left( \alpha \left( \mu \right) \right) =r\left( \tau \right) 
\text{, and}
\end{equation*}%
\begin{align*}
\left( \eta \tau \right) \omega & =\big(\alpha \left( \lambda \right) \alpha
\left( \mu \right) \big)\beta \left( \nu \right) =\alpha \left( \lambda \mu
\right) \beta \left( \nu \right) =\beta \left( \left( \lambda \mu \right)
\nu \right) \\
& =\beta \left( \lambda \left( \mu \nu \right) \right) =\alpha \left(
\lambda \right) \beta \left( \mu \nu \right) =\alpha \left( \lambda \right) %
\big(\alpha \left( \mu \right) \beta \left( \nu \right) \big)=\eta \left(
\tau \omega \right) \text{.}
\end{align*}%
Thus, $T\Lambda $ is a countable category, as claimed.

Now we show that $d$ is a functor. Note that both $T\Lambda $ and $\mathbb{N}%
^{k}$ are categories. First take object $x\in T\Lambda ^{0}$, then $d\left(
x\right) =0$ is an object in category $\mathbb{N}^{k}$. Next take morphisms $%
\tau ,\omega \in T\Lambda $ with $s\left( \tau \right) =r\left( \omega
\right) $. Then by definition of $d$,%
\begin{equation*}
d\left( \tau \omega \right) =d\left( \tau \right) +d\left( \omega \right) 
\text{.}
\end{equation*}%
Hence, $d$ is a functor.

To show that $d$ satisfies the factorisation property, take $\omega \in
T\Lambda $ and $m,n\in \mathbb{N}^{k}$ such that $d\left( \omega \right)
=m+n $. By definition, $\omega $ is either $\alpha \left( \lambda \right) $
or $\beta \left( \lambda \right) $ for some path $\lambda $ in $\Lambda $.
In both cases, there exist paths $\mu ,\nu $ in $\Lambda $ such that $%
\lambda =\mu \nu $, $d\left( \mu \right) =m$, and $d\left( \nu \right) =n$.
Then, we have $d\left( \alpha \left( \mu \right) \right) =m$, $d\left(
\alpha \left( \nu \right) \right) =d\left( \beta \left( \nu \right) \right)
=n$, and $\omega $ is either equal to $\alpha \left( \mu \right) \alpha
\left( \nu \right) $ or $\alpha \left( \mu \right) \beta \left( \nu \right) $%
. Therefore, the existence of factorisation is guaranteed.

Now we show that the factorisation is unique. First suppose $\omega =\alpha
\left( \mu \right) \alpha \left( \nu \right) =\alpha \left( \mu ^{\prime
}\right) \alpha \left( \nu ^{\prime }\right) $ where $d\left( \alpha \left(
\mu \right) \right) =d\left( \alpha \left( \mu ^{\prime }\right) \right) $
and $d\left( \alpha \left( \nu \right) \right) =d\left( \alpha \left( \nu
^{\prime }\right) \right) $. We consider paths $\lambda =\mu \nu $ and $%
\lambda ^{\prime }=\mu ^{\prime }\nu ^{\prime }$. Since $\alpha \left(
\lambda \right) =\omega =\alpha \left( \lambda ^{\prime }\right) $, then $%
\lambda =\lambda ^{\prime }$. This implies $\mu =\mu ^{\prime }$ and $\nu
=\nu ^{\prime }$ based on the uniquness of factorisation in $\Lambda $. Then 
$\alpha \left( \mu \right) =\alpha \left( \mu ^{\prime }\right) $ and $%
\alpha \left( \nu \right) =\alpha \left( \nu ^{\prime }\right) $. For the
case $\omega =\alpha \left( \mu \right) \beta \left( \nu \right) $, we get
the same result by using the same argument. The conclusion follows.
\end{proof}

\begin{remark}
For a directed graph $E$ (that is, for $k=1$), the graph $TE$ was
constructed by Muhly and Tomforde \cite[Definition 3.6]{MT04} (denoted $%
E_{V} $), and by Sims \cite[Section 3]{S10} (denoted $\widetilde{E}$). Our
notation follows that of Sims because we want to distinguish between paths
in $T\Lambda $ (denoted $\alpha \left( \lambda \right) $ and $\beta \left(
\lambda \right) $) and those in $\Lambda $ (denoted $\lambda $).
\end{remark}

\begin{remark}
\label{locally-convex}Every vertex $\beta \left( v\right) $ satisfies $\beta
\left( v\right) T\Lambda ^{1}=\emptyset $. Then if $\Lambda $ has a vertex $%
v $ which receives edges $e,f$ with $d\left( e\right) \neq d\left( f\right) $%
, then there is no edge $g\in \beta \left( s\left( e\right) \right) \Lambda
^{d\left( f\right) }$ (or $g\in \alpha \left( s\left( e\right) \right)
\Lambda ^{d\left( f\right) }$ if $s\left( e\right) \Lambda =\emptyset $),
and hence $\Lambda $ is not locally convex.
\end{remark}

The following lemma tells about properties of the $k$-graph $T\Lambda $.

\begin{proposition}
\label{characterisation-of-tlambda}Let $\Lambda $ be a row-finite $k$-graph
and $T\Lambda $ be the $k$-graph as in Proposition \ref{tlambda}. Then,

\begin{enumerate}
\item[(a)] $T\Lambda $ is row-finite.

\item[(b)] $T\Lambda $ is aperiodic.
\end{enumerate}
\end{proposition}

\begin{proof}
To show part (a), take $x\in T\Lambda ^{0}$. If $x=\beta \left( v\right) $
for some $v\in \Lambda ^{0}$, then $xT\Lambda ^{1}=\emptyset $ by Remark \ref%
{locally-convex}. Suppose $x=\alpha \left( v\right) $ for some $v\in \Lambda
^{0}$. If $v\Lambda ^{1}=\emptyset $, then $xT\Lambda ^{1}=\emptyset $.
Otherwise, for $1\leq i\leq k$ such that $v\Lambda ^{e_{i}}\neq \emptyset $,
we have 
\begin{equation*}
\left\vert xT\Lambda ^{e_{i}}\right\vert \leq 2\left\vert v\Lambda
^{e_{i}}\right\vert \text{,}
\end{equation*}
which is finite.

For part (b), take $\tau ,\omega \in T\Lambda $ such that $\tau \neq \omega $
and $s\left( \tau \right) =s\left( \omega \right) $. We have to show there
exists $\eta \in s\left( \tau \right) T\Lambda $ such that $\func{MCE}\left(
\tau \eta ,\omega \eta \right) =\emptyset $. If $s\left( \tau \right) =\beta
\left( v\right) $ for some $v\in \Lambda ^{0}$, then choose $\eta =\beta
\left( v\right) $ and $\func{MCE}\left( \tau \eta ,\omega \eta \right)
=\emptyset $. So suppose $s\left( \tau \right) =\alpha \left( v\right) $ for
some $v\in \Lambda ^{0}$. If $v\Lambda ^{1}=\emptyset $, then choose $\eta
=\alpha \left( v\right) $ and $\func{MCE}\left( \tau \eta ,\omega \eta
\right) =\emptyset $. Suppose $v\Lambda ^{1}\neq \emptyset $. Take $e\in
v\Lambda ^{1}$. If $s\left( e\right) \Lambda ^{1}=\emptyset $, then choose $%
\eta =\alpha \left( e\right) $ and $\func{MCE}\left( \tau \eta ,\omega \eta
\right) =\emptyset $. Otherwise, we have $s\left( e\right) \Lambda ^{1}\neq
\emptyset $. Then choose $\eta =\beta \left( e\right) $ and $\func{MCE}%
\left( \tau \eta ,\omega \eta \right) =\emptyset $. Hence, $T\Lambda $ is
aperiodic.
\end{proof}

\section{Realising $TC^{\ast }\left( \Lambda \right) $ as a Cuntz-Krieger
algebra}

\label{tc-and-c}

Let $\Lambda $ be a row-finite $k$-graph and $T\Lambda $ be the $k$-graph as
in Lemma \ref{tlambda}. In this Section, we show that $TC^{\ast }\left(
\Lambda \right) $ is isomorphic to $C^{\ast }\left( T\Lambda \right) $.

\begin{theorem}
\label{isomorphism}Let $\Lambda $ be a row-finite $k$-graph and $T\Lambda $
be the $k$-graph as in Proposition \ref{tlambda}. Let $\left\{ t_{\lambda
}:\lambda \in \Lambda \right\} $ be the universal Toeplitz-Cuntz-Krieger $%
\Lambda $-family and $\left\{ s_{\omega }:\omega \in T\Lambda \right\} $ be
the universal Cuntz-Krieger $\,T\Lambda $-family. For $\lambda \in \Lambda $%
, let%
\begin{equation*}
T_{\lambda }:=%
\begin{cases}
s_{\alpha \left( \lambda \right) }+s_{\beta \left( \lambda \right) } & \text{%
if }s\left( \lambda \right) \Lambda ^{1}\neq \emptyset \\ 
s_{\alpha \left( \lambda \right) } & \text{if }s\left( \lambda \right)
\Lambda ^{1}=\emptyset \text{.}%
\end{cases}%
\end{equation*}%
Then there is an isomorphism $\phi _{T}:TC^{\ast }\left( \Lambda \right)
\rightarrow C^{\ast }(T\Lambda )$ satisfying $\phi _{T}\left( t_{\lambda
}\right) =T_{\lambda }$ for every $\lambda \in \Lambda $.

Furthermore, $s_{\alpha \left( \lambda \right) }=\phi _{T}\left( t_{\lambda
}\right) $ if $s\left( \lambda \right) \Lambda ^{1}=\emptyset $. Meanwhile,
if $s\left( \lambda \right) \Lambda ^{1}\neq \emptyset $, we have $s_{\alpha
\left( \lambda \right) }=\phi _{T}\left( t_{\lambda }-t_{\lambda
}\prod_{e\in v\Lambda ^{1}}(t_{v}-t_{e}t_{e}^{\ast })\right) $ and $s_{\beta
\left( \lambda \right) }=\phi _{T}\left( t_{\lambda }\prod_{e\in v\Lambda
^{1}}(t_{v}-t_{e}t_{e}^{\ast })\right) $.
\end{theorem}

\begin{proof}[Proof that $\left\{ T_{\protect\lambda }:\protect\lambda \in
\Lambda \right\} $ is a Toeplitz-Cuntz-Krieger $\Lambda $-family]
To avoid an argument by cases, for $\lambda \in \Lambda $ with $s\left(
\lambda \right) \Lambda ^{1}=\emptyset $, we write $s_{\beta \left( \lambda
\right) }:=0$, so \ that 
\begin{equation*}
T_{\lambda }=s_{\alpha \left( \lambda \right) }+s_{\beta \left( \lambda
\right) }\text{.}
\end{equation*}

First, we want to show $\left\{ T_{\lambda }:\lambda \in \Lambda \right\} $
is a Toeplitz-Cuntz-Krieger $\Lambda $-family in $C^{\ast }(T\Lambda )$. For
(TCK1), take $v\in \Lambda ^{0}$. Since $\left\{ s_{\alpha \left( v\right)
}\right\} \cup \left\{ s_{\beta \left( v\right) }\right\} $ are mutually
orthogonal projections, then $T_{v}$ is a projection. Meanwhile, for $v,w\in
\Lambda ^{0}$ with $v\neq w$,%
\begin{equation*}
T_{v}T_{w}=s_{\alpha \left( v\right) }s_{\alpha \left( w\right) }+s_{\alpha
\left( v\right) }s_{\beta \left( w\right) }+s_{\beta \left( v\right)
}s_{\alpha \left( w\right) }+s_{\beta \left( v\right) }s_{\beta \left(
w\right) }=0\text{.}
\end{equation*}

Next we show (TCK2). Take $\mu ,\nu \in \Lambda $ where $s\left( \mu \right)
=r\left( \nu \right) $. Then%
\begin{equation*}
T_{\mu }T_{\nu }=s_{\alpha \left( \mu \right) }s_{\alpha \left( \nu \right)
}+s_{\alpha \left( \mu \right) }s_{\beta \left( \nu \right) }+s_{\beta
\left( \mu \right) }s_{\alpha \left( \nu \right) }+s_{\beta \left( \mu
\right) }s_{\beta \left( \nu \right) }\text{.}
\end{equation*}%
If $\nu $ is a vertex, the middle terms vanish and we get%
\begin{equation*}
T_{\mu }T_{\nu }=s_{\alpha \left( \mu \right) }+s_{\beta \left( \mu \right)
}=T_{\mu }\text{,}
\end{equation*}%
as required. Otherwise, the last two terms vanish and we get 
\begin{equation*}
T_{\mu }T_{\nu }=s_{\alpha \left( \mu \right) }s_{\alpha \left( \nu \right)
}+s_{\alpha \left( \mu \right) }s_{\beta \left( \nu \right) }=s_{\alpha
\left( \mu \nu \right) }+s_{\beta \left( \mu \nu \right) }=T_{\mu \nu }\text{%
,}
\end{equation*}%
which is (TCK2).

To show (TCK3), take $\lambda ,\mu \in \Lambda $. Then%
\begin{equation}
T_{\lambda }^{\ast }T_{\mu }=s_{\alpha \left( \lambda \right) }^{\ast
}s_{\alpha \left( \mu \right) }+s_{\alpha \left( \lambda \right) }^{\ast
}s_{\beta \left( \mu \right) }+s_{\beta \left( \lambda \right) }^{\ast
}s_{\alpha \left( \mu \right) }+s_{\beta \left( \lambda \right) }^{\ast
}s_{\beta \left( \mu \right) }\text{.}  \label{eq}
\end{equation}%
We give separate arguments for $\Lambda ^{\min }\left( \lambda ,\mu \right)
=\emptyset $ and $\Lambda ^{\min }\left( \lambda ,\mu \right) \neq \emptyset 
$. For case $\Lambda ^{\min }\left( \lambda ,\mu \right) =\emptyset $, we
have 
\begin{align*}
\emptyset & =T\Lambda ^{\min }\left( \alpha \left( \lambda \right) ,\alpha
\left( \mu \right) \right) =T\Lambda ^{\min }\left( \alpha \left( \lambda
\right) ,\beta \left( \mu \right) \right) \\
& =T\Lambda ^{\min }\left( \beta \left( \lambda \right) ,\alpha \left( \mu
\right) \right) =T\Lambda ^{\min }\left( \beta \left( \lambda \right) ,\beta
\left( \mu \right) \right) \text{.}
\end{align*}%
Hence, $s_{\alpha \left( \lambda \right) }^{\ast }s_{\alpha \left( \mu
\right) }=s_{\alpha \left( \lambda \right) }^{\ast }s_{\beta \left( \mu
\right) }=s_{\beta \left( \lambda \right) }^{\ast }s_{\alpha \left( \mu
\right) }=s_{\beta \left( \lambda \right) }^{\ast }s_{\beta \left( \mu
\right) }=0$ and then Equation \ref{eq} becomes 
\begin{equation*}
T_{\lambda }^{\ast }T_{\mu }=0=\sum_{(\lambda ^{\prime },\mu ^{\prime })\in
\Lambda ^{\min }\left( \lambda ,\mu \right) }T_{\lambda ^{\prime }}T_{\mu
^{\prime }}^{\ast }\text{.}
\end{equation*}

Now suppose $\Lambda ^{\min }\left( \lambda ,\mu \right) \neq \emptyset $.
Take $\left( a,b\right) \in \Lambda ^{\min }\left( \lambda ,\mu \right) $.
We consider several cases: whether $a$ equals $s\left( \lambda \right) $
and/or $b$ equals $s\left( \mu \right) $. First suppose $a=s\left( \lambda
\right) $ and $b=s\left( \mu \right) $. So $\lambda =\lambda s\left( \lambda
\right) =\mu s\left( \mu \right) =\mu $. Because $\alpha \left( \lambda
\right) $ and $\beta \left( \lambda \right) $ are paths with the same degree
and different sources, then $T\Lambda ^{\min }\left( \alpha \left( \lambda
\right) ,\beta \left( \lambda \right) \right) =\emptyset $. Thus,%
\begin{equation*}
s_{\beta \left( \lambda \right) }^{\ast }s_{\alpha \left( \lambda \right)
}=0=s_{\alpha \left( \lambda \right) }^{\ast }s_{\beta \left( \lambda
\right) }
\end{equation*}%
and Equation \ref{eq} becomes 
\begin{align*}
T_{\lambda }^{\ast }T_{\lambda }& =s_{\alpha \left( \lambda \right) }^{\ast
}s_{\alpha \left( \lambda \right) }+s_{\beta \left( \lambda \right) }^{\ast
}s_{\beta \left( \lambda \right) } \\
& =s_{s\left( \alpha \left( \lambda \right) \right) }+s_{s\left( \beta
\left( \lambda \right) \right) }=s_{\alpha \left( s\left( \lambda \right)
\right) }+s_{\beta \left( s\left( \lambda \right) \right) } \\
& =T_{_{s\left( \lambda \right) }}=T_{_{s\left( \lambda \right)
}}T_{_{s\left( \lambda \right) }}^{\ast } \\
& =\sum_{(\lambda ^{\prime },\mu ^{\prime })\in \Lambda ^{\min }\left(
\lambda ,\lambda \right) }T_{\lambda ^{\prime }}T_{\mu ^{\prime }}^{\ast }%
\text{ (since }\Lambda ^{\min }\left( \lambda ,\lambda \right) =s\left(
\lambda \right) \text{)}.
\end{align*}

Next suppose $a=s\left( \lambda \right) $ and $b\neq s\left( \mu \right) $.
Then $\lambda =\mu b$ and 
\begin{equation*}
T\Lambda ^{\min }\left( \alpha \left( \lambda \right) ,\beta \left( \mu
\right) \right) =\emptyset =T\Lambda ^{\min }\left( \beta \left( \lambda
\right) ,\beta \left( \mu \right) \right)
\end{equation*}%
since $s\left( \beta \left( \mu \right) \right) T\Lambda ^{1}=\emptyset $.
Hence%
\begin{equation*}
s_{\alpha \left( \lambda \right) }^{\ast }s_{\beta \left( \mu \right)
}=0=s_{\beta \left( \lambda \right) }^{\ast }s_{\beta \left( \mu \right) }
\end{equation*}%
and Equation \ref{eq} becomes%
\begin{equation*}
T_{\lambda }^{\ast }T_{\mu }=s_{\alpha \left( \lambda \right) }^{\ast
}s_{\alpha \left( \mu \right) }+s_{\beta \left( \lambda \right) }^{\ast
}s_{\alpha \left( \mu \right) }\text{.}
\end{equation*}%
Every $\left( \alpha \left( s\left( \lambda \right) \right) ,\eta \right)
\in T\Lambda ^{\min }\left( \alpha \left( \lambda \right) ,\alpha \left( \mu
\right) \right) $ has $\eta =\alpha \left( \mu ^{\prime }\right) $ with $%
(s\left( \lambda \right) ,\mu ^{\prime })\in \Lambda ^{\min }\left( \lambda
,\mu \right) $. Similarly, every $\left( \beta \left( s\left( \lambda
\right) \right) ,\eta \right) \in T\Lambda ^{\min }\left( \beta \left(
\lambda \right) ,\alpha \left( \mu \right) \right) $ has $\eta =\beta \left(
\mu ^{\prime }\right) $ with $(s\left( \lambda \right) ,\mu ^{\prime })\in
\Lambda ^{\min }\left( \lambda ,\mu \right) $. Thus, by using (TCK3) in $%
C^{\ast }\left( T\Lambda \right) $, 
\begin{align*}
T_{\lambda }^{\ast }T_{\mu }& =s_{\alpha \left( \lambda \right) }^{\ast
}s_{\alpha \left( \mu \right) }+s_{\beta \left( \lambda \right) }^{\ast
}s_{\alpha \left( \mu \right) } \\
& =\sum_{\left( \alpha \left( s\left( \lambda \right) \right) ,\eta \right)
\in T\Lambda ^{\min }\left( \alpha \left( \lambda \right) ,\alpha \left( \mu
\right) \right) }s_{\alpha \left( s\left( \lambda \right) \right) }s_{\eta
}^{\ast }+\sum_{\left( \beta \left( s\left( \lambda \right) \right) ,\eta
\right) \in T\Lambda ^{\min }\left( \beta \left( \lambda \right) ,\alpha
\left( \mu \right) \right) }s_{\beta \left( s\left( \lambda \right) \right)
}s_{\eta }^{\ast } \\
& =\sum_{(s\left( \lambda \right) ,\mu ^{\prime })\in \Lambda ^{\min }\left(
\lambda ,\mu \right) }s_{\alpha \left( s\left( \lambda \right) \right)
}s_{\alpha (\mu ^{\prime })}^{\ast }+\sum_{(s\left( \lambda \right) ,\mu
^{\prime })\in \Lambda ^{\min }\left( \lambda ,\mu \right) }s_{\beta \left(
s\left( \lambda \right) \right) }s_{\beta (\mu ^{\prime })}^{\ast } \\
& =\sum_{(s\left( \lambda \right) ,\mu ^{\prime })\in \Lambda ^{\min }\left(
\lambda ,\mu \right) }(s_{\alpha \left( s\left( \lambda \right) \right)
}s_{\alpha (\mu ^{\prime })}^{\ast }+s_{\beta \left( s\left( \lambda \right)
\right) }s_{\beta (\mu ^{\prime })}^{\ast }) \\
& =\sum_{(s\left( \lambda \right) ,\mu ^{\prime })\in \Lambda ^{\min }\left(
\lambda ,\mu \right) }(s_{\alpha \left( s\left( \lambda \right) \right)
}+s_{\beta \left( s\left( \lambda \right) \right) })(s_{\alpha (\mu ^{\prime
})}^{\ast }+s_{\beta (\mu ^{\prime })}^{\ast }) \\
& =\sum_{(s\left( \lambda \right) ,\mu ^{\prime })\in \Lambda ^{\min }\left(
\lambda ,\mu \right) }T_{s\left( \lambda \right) }T_{\mu ^{\prime }}^{\ast
}=\sum_{(\lambda ^{\prime },\mu ^{\prime })\in \Lambda ^{\min }\left(
\lambda ,\mu \right) }T_{\lambda ^{\prime }}T_{\mu ^{\prime }}^{\ast }\text{.%
}
\end{align*}

By taking adjoints, we deduce (TCK3) when $a\neq s\left( \lambda \right) $
and $b=s\left( \mu \right) $.

Now we consider the last case, which is $a\neq s\left( \lambda \right) $ and 
$b\neq s\left( \mu \right) $. This means we have neither $\lambda =\mu b$
nor $\mu =\lambda a$. Hence, 
\begin{equation*}
T\Lambda ^{\min }\left( \alpha \left( \lambda \right) ,\beta \left( \mu
\right) \right) =T\Lambda ^{\min }\left( \beta \left( \lambda \right)
,\alpha \left( \mu \right) \right) =T\Lambda ^{\min }\left( \beta \left(
\lambda \right) ,\beta \left( \mu \right) \right) =\emptyset
\end{equation*}%
since $s\left( \beta \left( \lambda \right) \right) T\Lambda ^{1}=\emptyset
=s\left( \beta \left( \mu \right) \right) T\Lambda ^{1}=\emptyset $. Hence,%
\begin{equation*}
s_{\alpha \left( \lambda \right) }^{\ast }s_{\beta \left( \mu \right)
}=s_{\beta \left( \lambda \right) }^{\ast }s_{\alpha \left( \mu \right)
}=s_{\beta \left( \lambda \right) }^{\ast }s_{\beta \left( \mu \right) }=0%
\text{.}
\end{equation*}%
On the other hand, we have%
\begin{equation*}
T\Lambda ^{\min }\left( \alpha \left( \lambda \right) ,\alpha \left( \mu
\right) \right) =\left\{ \left( \alpha \left( \lambda ^{\prime }\right)
,\alpha \left( \mu ^{\prime }\right) \right) ,\left( \beta \left( \lambda
^{\prime }\right) ,\beta \left( \mu ^{\prime }\right) \right) :(\lambda
^{\prime },\mu ^{\prime })\in \Lambda ^{\min }\left( \lambda ,\mu \right)
\right\} \text{.}
\end{equation*}%
Therefore, Equation \ref{eq} becomes%
\begin{align*}
T_{\lambda }^{\ast }T_{\mu }& =s_{\alpha \left( \lambda \right) }^{\ast
}s_{\alpha \left( \mu \right) }=\sum_{\left( \omega ,\eta \right) \in
T\Lambda ^{\min }\left( \alpha \left( \lambda \right) ,\alpha \left( \mu
\right) \right) }s_{\omega }s_{\eta }^{\ast } \\
& =\sum_{(\lambda ^{\prime },\mu ^{\prime })\in \Lambda ^{\min }\left(
\lambda ,\mu \right) }(s_{\alpha (\lambda ^{\prime })}s_{\alpha (\mu
^{\prime })}^{\ast }+s_{\beta (\lambda ^{\prime })}s_{\beta (\mu ^{\prime
})}^{\ast }) \\
& =\sum_{(\lambda ^{\prime },\mu ^{\prime })\in \Lambda ^{\min }\left(
\lambda ,\mu \right) }(s_{\alpha (\lambda ^{\prime })}+s_{\beta (\lambda
^{\prime })})(s_{\alpha (\mu ^{\prime })}^{\ast }+s_{\beta (\mu ^{\prime
})}^{\ast }) \\
& =\sum_{(\lambda ^{\prime },\mu ^{\prime })\in \Lambda ^{\min }\left(
\lambda ,\mu \right) }T_{\lambda ^{\prime }}T_{\mu ^{\prime }}^{\ast }\text{.%
}
\end{align*}

So for all cases, we have%
\begin{equation*}
T_{\lambda }^{\ast }T_{\mu }=\sum_{(\lambda ^{\prime },\mu ^{\prime })\in
\Lambda ^{\min }\left( \lambda ,\mu \right) }T_{\lambda ^{\prime }}T_{\mu
^{\prime }}^{\ast }
\end{equation*}%
and $\left\{ T_{\lambda }:\lambda \in \Lambda \right\} $ satisfies (TCK3).
\end{proof}

\begin{proof}[Proof that $\protect\phi _{T}$ is injective]
Now the\ universal property of $TC^{\ast }\left( \Lambda \right) $ gives a
homomorphism $\phi _{T}:TC^{\ast }\left( \Lambda \right) \rightarrow C^{\ast
}(T\Lambda )$ satisfying $\phi _{T}\left( t_{\lambda }\right) =T_{\lambda }$
for every $\lambda \in \Lambda $.

We show the injectivity of $\phi _{T}~$by using Theorem \ref%
{uniqueness-theorem-of-toeplitz}. Take $v\in \Lambda ^{0}$. We show 
\begin{equation*}
\prod_{e\in v\Lambda ^{1}}\left( T_{v}-T_{e}T_{e}^{\ast }\right) \neq 0.
\end{equation*}
First suppose $v\Lambda ^{1}\neq \emptyset $. Take $1\leq i\leq k$ such that 
$v\Lambda ^{e_{i}}\neq \emptyset $. We claim 
\begin{equation*}
\prod\limits_{e\in v\Lambda ^{e_{i}}}\left( T_{v}-T_{e}T_{e}^{\ast }\right)
\geq s_{\beta \left( v\right) }\text{.}
\end{equation*}%
Since $v\Lambda ^{e_{i}}\neq \emptyset $, then $\alpha \left( v\right)
T\Lambda ^{e_{i}}\neq \emptyset $ and by \cite[Lemma 2.7 (iii)]{RSY04},%
\begin{align*}
s_{\alpha \left( v\right) }& \geq \sum_{g\in \alpha \left( v\right) T\Lambda
^{e_{i}}}s_{g}s_{g}^{\ast } \\
& =\sum_{e\in v\Lambda ^{e_{i}}}s_{\alpha \left( e\right) }s_{\alpha \left(
e\right) }^{\ast }+\sum_{\substack{ e\in v\Lambda ^{e_{i}}  \\ s\left(
e\right) \Lambda ^{1}\neq \emptyset }}s_{\beta \left( e\right) }s_{\beta
\left( e\right) }^{\ast } \\
& =\sum_{\substack{ e\in v\Lambda ^{e_{i}}  \\ s\left( e\right) \Lambda
^{1}\neq \emptyset }}\left( s_{\alpha \left( e\right) }s_{\alpha \left(
e\right) }^{\ast }+s_{\beta \left( e\right) }s_{\beta \left( e\right)
}^{\ast }\right) +\sum_{\substack{ e\in v\Lambda ^{e_{i}}  \\ s\left(
e\right) \Lambda ^{1}=\emptyset }}s_{\alpha \left( e\right) }s_{\alpha
\left( e\right) }^{\ast } \\
& =\sum_{\substack{ e\in v\Lambda ^{e_{i}}  \\ s\left( e\right) \Lambda
^{1}\neq \emptyset }}T_{e}T_{e}^{\ast }+\sum_{\substack{ e\in v\Lambda
^{e_{i}}  \\ s\left( e\right) \Lambda ^{1}=\emptyset }}T_{e}T_{e}^{\ast } \\
& =\sum_{e\in v\Lambda ^{e_{i}}}T_{e}T_{e}^{\ast }\text{.}
\end{align*}%
Meanwhile, since every $e\in v\Lambda ^{e_{i}}$ has the same degree, 
\begin{align*}
\prod\limits_{e\in v\Lambda ^{e_{i}}}\left( T_{v}-T_{e}T_{e}^{\ast }\right)
& =T_{v}-\sum_{e\in v\Lambda ^{e_{i}}}T_{e}T_{e}^{\ast } \\
& =\left( s_{\alpha \left( v\right) }+s_{\beta \left( v\right) }\right)
-\sum_{e\in v\Lambda ^{e_{i}}}T_{e}T_{e}^{\ast } \\
& =s_{\beta \left( v\right) }+\Big(s_{\alpha \left( v\right) }-\sum_{e\in
v\Lambda ^{e_{i}}}T_{e}T_{e}^{\ast }\Big) \\
& \geq s_{\beta \left( v\right) }\text{,}
\end{align*}%
as claimed. This claim implies

\begin{equation*}
\prod\limits_{e\in v\Lambda ^{1}}\left( T_{v}-T_{e}T_{e}^{\ast }\right) \geq
\prod_{\left\{ i:v\Lambda ^{e_{i}}\neq \emptyset \right\} }s_{\beta \left(
v\right) }=s_{\beta \left( v\right) }\neq 0
\end{equation*}%
since $v\Lambda ^{1}\neq \emptyset $, as required.

Finally, for $v\in \Lambda ^{0}$ with $v\Lambda ^{1}=\emptyset $, we have%
\begin{equation*}
T_{v}=s_{\alpha \left( v\right) }\neq 0\text{.}
\end{equation*}%
Hence, by Theorem \ref{uniqueness-theorem-of-toeplitz}, $\phi _{T}$ is
injective.
\end{proof}

\begin{proof}[Proof that $\protect\phi _{T}$ is surjective]
Now we show the surjectivity of $\phi _{T}$. Since $C^{\ast }(T\Lambda )$ is
generated by $\left\{ s_{\tau }:\tau \in T\Lambda \right\} $, then it
suffices to show that for every $\tau \in T\Lambda $, $s_{\tau }\in \func{im}%
\left( \phi _{T}\right) $. Recall that for every $\tau \in T\Lambda $, $%
s_{\tau }$ is either $s_{\alpha \left( \lambda \right) }$ or $s_{\beta
\left( \lambda \right) }$ for some $\lambda \in \Lambda $.

Take $v\in \Lambda ^{0}$. First we show $s_{\alpha \left( v\right) }$ and $%
s_{\beta \left( v\right) }$ (if it exists) belong to $\func{im}\left( \phi
_{T}\right) $. If $v\Lambda ^{1}=\emptyset $, then%
\begin{equation*}
s_{\alpha \left( v\right) }=T_{v}\in \func{im}\left( \phi _{T}\right) \text{.%
}
\end{equation*}%
Next suppose $v\Lambda ^{1}\neq \emptyset $. First we show that $s_{\beta
\left( v\right) }=\prod_{e\in v\Lambda ^{1}}\left( T_{v}-T_{e}T_{e}^{\ast
}\right) $. Note that for every $f\in \alpha \left( v\right) T\Lambda ^{1}$,
the projection $s_{\alpha \left( v\right) }-s_{f}s_{f}^{\ast }\leq s_{\alpha
\left( v\right) }$ is othogonal to $s_{\beta \left( v\right) }$. This implies%
\begin{eqnarray*}
\prod_{f\in \alpha \left( v\right) T\Lambda ^{1}}((s_{\alpha \left( v\right)
}+s_{\beta \left( v\right) })-s_{f}s_{f}^{\ast }) &=&s_{\beta \left(
v\right) }+\prod_{f\in \alpha \left( v\right) T\Lambda ^{1}}(s_{\alpha
\left( v\right) }-s_{f}s_{f}^{\ast }) \\
&=&s_{\beta \left( v\right) }\text{,}
\end{eqnarray*}%
since $v\Lambda ^{1}$ is an exhaustive set. Hence,%
\begin{eqnarray*}
s_{\beta \left( v\right) } &=&\prod_{f\in \alpha \left( v\right) T\Lambda
^{1}}((s_{\alpha \left( v\right) }+s_{\beta \left( v\right)
})-s_{f}s_{f}^{\ast }) \\
&=&\prod_{e\in v\Lambda ^{1}}(T_{v}-s_{\alpha \left( e\right) }s_{\alpha
\left( e\right) }^{\ast })\prod_{\substack{ e\in v\Lambda ^{1}  \\ s\left(
e\right) \Lambda ^{1}\neq \emptyset }}(T_{v}-s_{\beta \left( e\right)
}s_{\beta \left( e\right) }^{\ast }) \\
&=&\prod_{\substack{ e\in v\Lambda ^{1}  \\ s\left( e\right) \Lambda
^{1}=\emptyset }}(T_{v}-s_{\alpha \left( e\right) }s_{\alpha \left( e\right)
}^{\ast })\prod_{\substack{ e\in v\Lambda ^{1}  \\ s\left( e\right) \Lambda
^{1}\neq \emptyset }}(T_{v}-s_{\alpha \left( e\right) }s_{\alpha \left(
e\right) }^{\ast })(T_{v}-s_{\beta \left( e\right) }s_{\beta \left( e\right)
}^{\ast }) \\
&=&\prod_{\substack{ e\in v\Lambda ^{1}  \\ s\left( e\right) \Lambda
^{1}=\emptyset }}(T_{v}-s_{\alpha \left( e\right) }s_{\alpha \left( e\right)
}^{\ast })\prod_{\substack{ e\in v\Lambda ^{1}  \\ s\left( e\right) \Lambda
^{1}\neq \emptyset }}(T_{v}-(s_{\alpha \left( e\right) }s_{\alpha \left(
e\right) }^{\ast }+s_{\beta \left( e\right) }s_{\beta \left( e\right)
}^{\ast })) \\
&=&\prod_{\substack{ e\in v\Lambda ^{1}  \\ s\left( e\right) \Lambda
^{1}=\emptyset }}(T_{v}-T_{e}T_{e}^{\ast })\prod_{\substack{ e\in v\Lambda
^{1}  \\ s\left( e\right) \Lambda ^{1}\neq \emptyset }}(T_{v}-T_{e}T_{e}^{%
\ast }) \\
&=&\prod_{e\in v\Lambda ^{1}}\left( T_{v}-T_{e}T_{e}^{\ast }\right) \text{,}
\end{eqnarray*}%
as required, and $s_{\beta \left( v\right) }$ belongs to $\func{im}\left(
\phi _{T}\right) $. Furthermore,%
\begin{equation*}
s_{\alpha \left( v\right) }=T_{v}-s_{\beta \left( v\right)
}=T_{v}-\prod_{e\in v\Lambda ^{1}}(T_{v}-T_{e}T_{e}^{\ast })\in \func{im}%
\left( \phi _{T}\right) \text{,}
\end{equation*}%
as required.

Now take $\lambda \in \Lambda $. We have to show $s_{\alpha \left( \lambda
\right) }$ and $s_{\beta \left( \lambda \right) }$ (if it exists) belong to $%
\func{im}\left( \phi _{T}\right) $. If $s\left( \lambda \right) \Lambda
^{1}=\emptyset $, then 
\begin{equation*}
s_{\alpha \left( \lambda \right) }=s_{\alpha \left( \lambda \right)
}s_{\alpha \left( s\left( \lambda \right) \right) }=T_{\lambda }T_{s\left(
\lambda \right) }=T_{\lambda }\in \func{im}\left( \phi _{T}\right) \text{.}
\end{equation*}%
Next suppose $s\left( \lambda \right) \Lambda ^{1}\neq \emptyset $. Then $%
s_{\beta \left( \lambda \right) }s_{\alpha \left( s\left( \lambda \right)
\right) }=0$ and $s_{\alpha \left( \lambda \right) }s_{\beta \left( s\left(
\lambda \right) \right) }=0$. Hence, 
\begin{align*}
s_{\alpha \left( \lambda \right) }& =s_{\alpha \left( \lambda \right)
}s_{\alpha \left( s\left( \lambda \right) \right) }=\left( s_{\alpha \left(
\lambda \right) }+s_{\beta \left( \lambda \right) }\right) s_{\alpha \left(
s\left( \lambda \right) \right) } \\
& =T_{\lambda }\Big(T_{s\left( \lambda \right) }-\prod_{e\in s\left( \lambda
\right) \Lambda ^{1}}(T_{s\left( \lambda \right) }-T_{e}T_{e}^{\ast })\Big)
\\
& =T_{\lambda }-T_{\lambda }\prod_{e\in s\left( \lambda \right) \Lambda
^{1}}(T_{s\left( \lambda \right) }-T_{e}T_{e}^{\ast })\in \func{im}\left(
\phi _{T}\right)
\end{align*}%
and%
\begin{align*}
s_{\beta \left( \lambda \right) }& =s_{\beta \left( \lambda \right)
}s_{\beta \left( s\left( \lambda \right) \right) }=\left( s_{\alpha \left(
\lambda \right) }+s_{\beta \left( \lambda \right) }\right) s_{\beta \left(
s\left( \lambda \right) \right) } \\
& =T_{\lambda }\prod_{e\in s\left( \lambda \right) \Lambda ^{1}}(T_{s\left(
\lambda \right) }-T_{e}T_{e}^{\ast })\in \func{im}\left( \phi _{T}\right) 
\text{.}
\end{align*}%
Therefore, $\phi _{T}$ is surjective and an isomorphism.
\end{proof}

\begin{corollary}
\label{isomorphism-2}Let $\Lambda $ be a row-finite $k$-graph and $T\Lambda $
be the $k$-graph as in Proposition \ref{tlambda}. Let $\left\{ t_{\lambda
}:\lambda \in \Lambda \right\} $ be the universal Toeplitz-Cuntz-Krieger $%
\Lambda $-family and $\left\{ s_{\omega }:\omega \in T\Lambda \right\} $ be
the universal Cuntz-Krieger $\,T\Lambda $-family. For $\tau \in T\Lambda $,
define%
\begin{equation*}
S_{\tau }:=%
\begin{cases}
t_{\lambda } & \text{if }\tau =\alpha \left( \lambda \right) \text{ with }%
s\left( \lambda \right) \Lambda ^{1}=\emptyset \\ 
t_{\lambda }-t_{\lambda }\prod_{e\in v\Lambda ^{1}}(t_{v}-t_{e}t_{e}^{\ast })
& \text{if }\tau =\alpha \left( \lambda \right) \text{ with }s\left( \lambda
\right) \Lambda ^{1}\neq \emptyset \\ 
t_{\lambda }\prod_{e\in v\Lambda ^{1}}(t_{v}-t_{e}t_{e}^{\ast }) & \text{if }%
\tau =\beta \left( \lambda \right) \text{ with }s\left( \lambda \right)
\Lambda ^{1}\neq \emptyset \text{.}%
\end{cases}%
\end{equation*}%
Suppose that $\phi _{T}:TC^{\ast }\left( \Lambda \right) \rightarrow C^{\ast
}(T\Lambda )$ is the isomorphism as in Theorem \ref{isomorphism}\ and $\pi
_{S}:C^{\ast }\left( T\Lambda \right) \rightarrow TC^{\ast }\left( \Lambda
\right) $ is the homomorphism such that $\pi _{S}\left( s_{\tau }\right)
=S_{\tau }$ for $\tau \in T\Lambda $. Then $\phi _{T}^{-1}=\pi _{S}$.
\end{corollary}

\begin{proof}
Take $\lambda \in \Lambda $. By Theorem \ref{isomorphism}, we get $\phi
_{T}^{-1}\left( s_{\alpha \left( \lambda \right) }\right) =t_{\lambda }$ if $%
s\left( \lambda \right) \Lambda ^{1}=\emptyset $. Meanwhile, if $s\left(
\lambda \right) \Lambda ^{1}\neq \emptyset $, by Theorem \ref{isomorphism},
we have $\phi _{T}^{-1}\left( s_{\alpha \left( \lambda \right) }\right)
=t_{\lambda }-t_{\lambda }\prod_{e\in v\Lambda ^{1}}(t_{v}-t_{e}t_{e}^{\ast
})$ and $\phi _{T}^{-1}\left( s_{\beta \left( \lambda \right) }\right)
=t_{\lambda }\prod_{e\in v\Lambda ^{1}}(t_{v}-t_{e}t_{e}^{\ast })$. Hence, $%
\phi _{T}^{-1}\left( s_{\tau }\right) =S_{\tau }$ for $\tau \in T\Lambda $.
This implies that $\left\{ S_{\tau }:\tau \in T\Lambda \right\} $ is a
Cuntz-Krieger $T\Lambda $-family, and then $\phi _{T}^{-1}=\pi _{S}$.
\end{proof}

\begin{remark}
\label{isomorphism-other-direction}Proposition \ref%
{characterisation-of-tlambda} says that $T\Lambda $ is always aperiodic, and
hence the Cuntz-Krieger uniqueness theorem always applies to $T\Lambda $.
This helps explain why no hypothesis on $\Lambda $ is required in the
uniquness theorem of \cite[Theorem 8.1]{RS05}. Indeed, we could have deduced
that theorem by applying the Cuntz-Krieger uniqueness theorem to $T\Lambda $%
. With our current proof of Theorem \ref{isomorphism}, this argument would
be circular, since we used \cite[Theorem 8.1]{RS05} in the proof of Theorem %
\ref{isomorphism}. However, we could prove Corollary \ref{isomorphism-2}
directly by showing that $\left\{ S_{\tau }:\tau \in T\Lambda \right\} $ is
a Cuntz-Krieger $T\Lambda $-family in $TC^{\ast }\left( \Lambda \right) $,
hence gives a homomorphism $\pi _{S}:C^{\ast }\left( T\Lambda \right)
\rightarrow TC^{\ast }\left( \Lambda \right) $, and using the Cuntz-Krieger
uniqueness theorem to see that $\pi _{S}$ is injective. Then we could deduce 
\cite[Theorem 8.1]{RS05} from Corollary \ref{isomorphism-2}, and this would
be a legitimate proof. We worked out the details of this approach, but it
seemed to require an extensive cases argument, and hence became
substantially more complicated.
\end{remark}

\end{document}